\documentclass[a4paper]{article}

\usepackage{amsmath}
\usepackage{amssymb}
\usepackage[english]{babel}
\usepackage{amsthm}
\usepackage[latin1]{inputenc}
\usepackage{graphicx}
\usepackage{subfigure}
\usepackage{makeidx}
\usepackage{verbatim}
\usepackage{xcolor}
\usepackage{marginnote}
\usepackage{mathrsfs}
\usepackage{hyperref}
\usepackage{tikz}
\usetikzlibrary{calc}

\newtheorem{theorem}{Theorem}
\newtheorem{lemma}{Lemma}
\newtheorem{proposition}{Proposition}
\newtheorem{corollary}{Corollary}
\newtheorem{remark}{Remark}
\newtheorem{conjecture}{Conjecture}
\theoremstyle{definition}
\newtheorem{definition}{Definition}

\def\D{\mathbb{D}}
\def\R{\mathbb{R}}
\def\Q{\mathbb{Q}}

\def\Z{\mathbb{Z}}

\begin{document}

\title{Spectral rigidity among ellipses, Bialy's conjecture and local extrema of Mather's beta function}
\author{Corentin Fierobe}
\date{}
\maketitle

\begin{abstract}
In this paper we prove Bialy's conjecture, which states that if two ellipses $\mathscr E$ and $\mathscr E'$ satisfy 
$\beta_{\mathscr E}(\rho_0)=\beta_{\mathscr E'}(\rho_0)$ and $\beta_{\mathscr E}(\rho_1)=\beta_{\mathscr E'}(\rho_1)$ 
for two distinct rotation numbers $\rho_0,\rho_1\in(0,1/2]$, then they must coincide up to isometries. 
We also prove that the same conclusion holds when only one rotation number is prescribed, provided the two ellipses have the same perimeter. 
Finally, we discuss consequences for local extremizers of Mather's beta function, building on a recent result of Baranzini, Bialy, and Sorrentino.
\end{abstract}

\section{Introduction}

In the title of a seminal 1966 paper \cite{Kac}, Kac asked the question \textit{"Can one hear the shape of a drum?"} which has since become famous. It asks whether the shape of a domain in $\R^d$ is uniquely determined by the spectrum of the Laplacian inside the domain with Dirichlet or Neumann boundary conditions.

This question found a similar formulation in billiards dynamics, where one studies the movement of a particle inside a domain $\Omega$, bouncing off its boundary according to the law of reflection \textit{angle of incidence = angle of reflection}. The spectrum of the Laplacian is replaced by the closure of the set of perimeters of periodic orbits, or \textit{length spectrum} $\mathscr L(\Omega)$, and one asks whether it determines the domain. Andersson and Melrose \cite{AM} proved that the two questions are related for generic strictly convex domains.

In the planar case, the question admits another, non-equivalent formulation by imposing a \textit{marking}. More precisely, given a strictly convex domain $\Omega\subset\R^2$ and a rational number $p/q\in[0,1/2]$, one can define
\[
\beta_{\Omega}\left(\frac{p}{q}\right) = -\frac{1}{q}L_{p/q}  
\]
where $L_{p/q}$ is the maximal perimeter of an orbit of rotation number $p/q$, where $q$ is the number of bounces and $p$ is the winding number of the orbit. The map $\beta_{\Omega}$, which extends to a convex map of the interval $[0,1/2]$ is called Mather beta function and has many interesting properties relating the length spectrum to billiard's dynamics. 

The question whether $\mathscr L(\Omega)$ or $\beta_{\Omega}$ alone allow recovering a strictly convex planar domain remains open, although some partial results are known. 
Generalizing \cite{MelroseIsospectral, OsgoodPhillipsSarnak1, OsgoodPhillipsSarnak2, OsgoodPhillipsSarnak3} to the length spectrum, \cite{Vig} proved that the sets of isospectral domains are compact in the $\mathscr C^{\infty}$-topology.
\cite{Siburg, SorDCDS} showed that Mather's beta function, resp. its Taylor coefficients, allow one to recognize disks from any other domain. In fact, one needs only the data of the domain's perimeter and the value of $\beta_{\Omega}$ at a generic rotation number; see \cite{BBS}. 
Recent developments towards the proof of Birkhoff's conjecture, see \cite{BialyMironov, KS} and in particular \cite[Theorem 2.3]{Bialyellipses}, have shown that \textit{one can recognize ellipses from any other centrally symmetric strictly convex domain with $\mathscr C^2$-smooth boundary using the values of $\beta_{\Omega}$ at rotation numbers in $(0,1/4]$.} In particular, the result holds among ellipses, but the proof works with the knowledge of $\beta_{\Omega}$ at an infinite family of rational rotation numbers. 

This problem can also be formulated in terms of deformations and has been studied in several works; see, for example, \cite{DKW, FKS_ellipse, HezariZelditch, HZ2}. We emphasize that the variational techniques developed in \cite{FKS_ellipse} for the study of Mather beta function inspired the results of the present paper. We refer to \cite{FKS_lecturenotes} for a recent overview of these problems.
\vspace{0.2cm}

In this paper, we are interested in studying the question of recovering $\Omega$ from the knowledge of $\beta_{\Omega}(\rho)$ on a \textit{finite} set of rotation numbers. This seems impossible to hold for general domains, hence we investigate this problem either among ellipses, or locally near a fixed general domain. This problem was recently investigated in two different works \cite{BBS, Bialyellipses}, and this paper addresses questions arising from these works in the two cases previously mentioned.

In \cite{Bialyellipses}, the author formulates an explicit expression of Mather's beta function associated to an ellipse, and provides remarkable computations that describe fully the billiard dynamics in an ellipse. He also proves using a geometric argument that two ellipses $\mathscr E$ and $\mathscr E'$ satisfying
\[
\beta_{\mathscr E}\left(\frac{p}{q}\right) = 
\beta_{\mathscr E'}\left(\frac{p}{q}\right)
\quad\text{and}\quad
\beta_{\mathscr E}\left(\frac{1}{2}\right) = 
\beta_{\mathscr E'}\left(\frac{1}{2}\right)
\]
for a given rational rotation number $p/q\in(0,1/2)$, then $\mathscr E$ and $\mathscr E'$ must be the same. He conjectures that the result holds if one replaces $p/q$ and $1/2$ by any two distinct rotation numbers:

\begin{conjecture}[Bialy \cite{Bialyellipses}]
Let $\rho_0$ and $\rho_1$ be two distinct rotation numbers in $(0,1/2]$. Two ellipses $\mathscr E$ and $\mathscr E'$ satisfying
\[
\beta_{\mathscr E}\left(\rho_0\right) = 
\beta_{\mathscr E'}\left(\rho_0\right)
\quad\text{and}\quad
\beta_{\mathscr E}\left(\rho_1\right) = 
\beta_{\mathscr E'}\left(\rho_1\right)
\]
are the same up to isometries.
\end{conjecture}

We will prove Bialy's conjecture:

\begin{theorem}
\label{thm:main_bialy_conj}
    Let $\rho_0,\rho_1\in(0,1/2]$ be distinct. Assume that $\mathscr E$ and $\mathscr E'$ are two ellipses satisfying
    \[
    \beta_{\mathscr E}(\rho_0) = \beta_{\mathscr E'}(\rho_0)
    \qquad\text{and}\qquad
    \beta_{\mathscr E}(\rho_1) = \beta_{\mathscr E'}(\rho_1).
    \]
    Then $\mathscr E=\mathscr E'$ up to isometries.
\end{theorem}

The proof shows that given $\rho_0\in(0,1/2]$ and an ellipse $\mathscr E$, there is an analytic one-parameter family of ellipses $(\mathscr E_e)_{e\in[0,1)}$ such that each ellipse $\mathscr E_e$ is the only ellipse of eccentricity $e$ satisfying $\beta_{\mathscr E_e}(\rho_0)=\beta_{\mathscr E}(\rho_0)$ up to isometries. Then it shows that for any other $\rho_1\neq\rho_0$, the map 
\[
e\in[0,1)\mapsto\beta_{\mathscr E_e}(\rho_1)
\]
is strictly monotone.\\

The same techniques allow us to show a similar result which can be described as follows. Given $p>0$ we consider the family of ellipses $(\mathscr E_e)_{e\in[0,1]}$ such that $\mathscr E_e$ is the unique ellipse of perimeter $p$ and eccentricity $e$ up to isometries. The case $e=0$ corresponds to a disk, and the case $e=1$ corresponds to the degenerate case of a flat ellipse. We prove the following result:

\begin{theorem}
\label{thm:main_cst_perimeter}
    Given $\rho\in(0,1/2]$, the map 
    \begin{equation}
    \label{eq:map_beta_ecc}
    e\in[0,1]\mapsto\beta_{\Omega_e}(\rho)
    \end{equation}
    is strictly decreasing.
\end{theorem}

In fact, if $\rho=0$ the map \eqref{eq:map_beta_ecc} is identically equal to $0$, hence we need to exclude it. We deduce immediately the following:

\begin{corollary}
\label{cor:main_cst_perimeter}
    Let $\rho\in(0,1/2]$. Assume that $\mathscr E$ and $\mathscr E'$ are two ellipses with the same perimeter such that
    \[
    \beta_{\mathscr E}(\rho) = \beta_{\mathscr E'}(\rho).
    \]
    Then $\mathscr E=\mathscr E'$ up to isometries.
\end{corollary}
\vspace{0.2cm}

Previous results can be interpreted in terms of the local maxima of Mather's beta function. Before giving an explicit formulation for this interpretation, let us recall a recent result by Bialy, Baranzini and Sorrentino, see \cite[Theorem 4.1]{BBS}, stating that for any rotation number $\rho$ outside a set of zero measure, one can characterize a given domain $\Omega$ as a disk uusing only the data of its perimeter $\partial\Omega$ and $\beta_{\Omega}$. More precisely, the result can be stated as follows, where we introduce $\D$ as the unit disk:

\begin{theorem}[Bialy-Baranzini-Sorrentino \cite{BBS}]
\label{thm:bbs}
Let $\Omega\subset\R^2$ be a strictly convex domain. Then
\begin{equation}
    \label{eq:inequality_disk}
    \beta_{\Omega}(\rho) \leq \frac{|\partial\Omega|}{2\pi}\beta_{\D}(\rho),\qquad \forall\rho\in\left[0,\frac{1}{2}\right].
\end{equation}
If the inequality in \eqref{eq:inequality_disk} is an equality at a certain $\rho$, then the billiard in $\Omega$ has an invariant curve of constant angle associated to $\rho$. Moreover, there is a set $\mathscr R\subset[0,1/2]$ containing $\Q\cap(0,1/2)$ whose complement is dense and of zero measure, such that if equality in \eqref{eq:inequality_disk} holds at a $\rho\in\mathscr R$, then $\Omega$ is a disk.
\end{theorem}

We recall what it means for a billiard map in a strictly convex domain $\Omega\subset\R^2$ to have an invariant curve of constant angle associated to $\rho\in[0,1/2]$. This notion was introduced and studied by Gutkin \cite{Gutkin} to study the so-called \textit{floating problem}.

The boundary of $\Omega$ is assumed to be $\mathscr C^2$-smooth. The billiard map $T_{\Omega}$ in $\Omega$ is defined on the set of pairs $(q,\varphi)\in\partial\Omega\times [0,\pi]$: 
in fact $T_{\Omega}(q,\varphi)=(q',\varphi')$ if and only if the oriented line starting at $q$ and making an angle $\varphi$ with the boundary intersects the boundary at another point $q'$, 
and bounces off it with an angle $\varphi'$. 
The map $T_{\Omega}$ is said to have \textit{an invariant curve of constant angle associated to $\rho$} if there is $\varphi_0\in[0,\pi]$ such that the set $\mathscr C :=\partial\Omega\times\{\varphi_0\}$ is invariant by $T$, and $\left.T_{\Omega}\right|_{\mathscr C}$ has rotation number $\rho$.

Disks are natural examples of billiard table with invariant curves of constant angle associated to any rotation number in $[0,1/2]$. Gutkin \cite[Corollary 2]{Gutkin} showed that there is a set, which is exactly the set $\mathscr R$ of Theorem \ref{thm:bbs}, such that the only domains having an invariant curve of constant angle associated to a $\rho\in\mathscr R$ are the circular billiards. He also showed that if $\rho\notin\mathscr R$, there are non-circular domains having an invariant curve of constant angle associated to $\rho$.

As a consequence of Theorem \ref{thm:bbs}, disks are global maximizer of the function 
\begin{equation}
\label{equ:beta_rho}
\beta(\rho):\Omega\in\mathscr D_p\mapsto\beta_{\Omega}(\rho)
\end{equation}
defined for a fixed $\rho$ over the set $\mathscr D_p$ of strictly convex domains of a given perimeter $p>0$. Moreover, the maximum is strict for all except a countable set of rotation numbers.

In this paper we also intend to address the question of the \textit{local} extrema of $\beta(\rho)$, to which Theorem \ref{thm:main_cst_perimeter} and Corollary \ref{cor:main_cst_perimeter} are related, see Corollary \ref{cor:ell_non_loc_max}. Given an integer $r>0$, a $\mathscr C^r$-smooth domain $\Omega$ is said to be a \textit{local $\mathscr C^r$-maximizer} (resp. \textit{local $\mathscr C^r$-minimizer}) of $\beta(\rho)$ if there exists a neighborhood $\mathcal U$ of $\Omega$ in the $\mathscr C^r$-smooth topology such that for any $\Omega'\in\mathcal U$,
    \[
    |\partial\Omega|=|\partial\Omega'|
    \qquad\Longrightarrow\qquad
    \beta_{\Omega'}(\rho) \leq \beta_{\Omega}(\rho)
    \qquad \text{resp. }
    \beta_{\Omega'}(\rho) \geq \beta_{\Omega}(\rho).
    \]
Given a domain $\Omega$, we denote by $\varrho$ a radius of curvature of the domain associated to a given parametrization by arc-length. A set $\mathscr S_r$ of domains with $\mathscr C^r$-smooth boundaries is said to be bounded if there is $C>0$ such that 
\[
\|\varrho\|_{\mathscr C^r}\leq C
\]
for any radii of curvature $\varrho$ of domains in $\mathscr S_r$. We show

\begin{theorem}
\label{thm:main_local_max}
There exists $r\geq 2$ with the following property. Given a $\mathscr C^r$-bounded set $\mathscr S_r$ of domains, there exists a positive measure set $\mathscr R'=\mathscr R'(\mathscr S_r)\subset(0,1/2]$ accumulating to $0$ such that if $\rho\in\mathscr R'$ then 
\begin{enumerate}
    \item $\beta(\rho)$ has no local $\mathscr C^r$-minimizers ;
    \item the only possible local $\mathscr C^r$-maximizers of $\beta(\rho)$ lying in $\mathscr S_r$ are disks.
\end{enumerate} 
\end{theorem}

\noindent Note that this result holds in particular when $\mathscr S_r$ is reduced to a single domain $\{\Omega\}$. In that case, $\mathscr R'$ is a positive measure set in $(0,1/2]$ -- depending on $\Omega$ -- such that for any $\rho\in\mathscr R'$, $\Omega$ is not a local $\mathscr C^r$-minimizer of $\beta(\rho)$, and not a local $\mathscr C^r$-maximizer of $\beta(\rho)$, unless $\Omega$ is a disk.

\begin{remark}
    $\mathscr R'$ contains so-called Diophantine numbers corresponding to the existence of KAM curves. We don't know whether $\mathscr R'$ contains other numbers, such as rational numbers. 
\end{remark}

\begin{remark}
    In Theorem \ref{thm:main_local_max}, the set $\mathscr S_r$ doesn't need to be open. A domain $\Omega\in\mathscr S_r$ is a $\mathscr C^r$-maximizers of $\beta(\rho)$ if it maximizes the map $\Omega'\mapsto\beta_{\Omega'}(\rho)$ among all nearby domains of the same perimeter as $\Omega$, but not necessarily lying in $\mathscr S_r$.
\end{remark}

Since ellipses are natural generalizations of disks, one may ask whether ellipses can be local maximizers of $\beta(\rho)$ for any given $\rho\neq 0$. Corollary \ref{cor:main_cst_perimeter} immediately shows that his is not the case for any non-circular ellipses:

\begin{corollary}
    \label{cor:ell_non_loc_max}
    Let $r\in\Z_{>0}\cup\{+\infty,\rho\}$ and $\rho\in(0,1/2]$. A non-circular ellipse is not a local $\mathscr C^r$-maximizer of the map 
    \[
    \beta(\rho):\Omega\mapsto\beta_{\Omega}(\rho).
    \]
\end{corollary}

\section{Acknowledgements}
The author is grateful to Alfonso Sorrentino for useful discussions. 
He acknowledges the support of the Italian Ministry of University and Research's PRIN 2022 grant Stability in Hamiltonian dynamics and beyond, as well as the Department of Excellence grant MathMod@TOV (2023-27) awarded to the Department of Mathematics of the University of Rome Tor Vergata.

\section{First variations of Mather's beta function}

The billiard map inside a strictly convex domain $\Omega$ acts on the space of oriented lines intersecting $\Omega$. 
Fix a point $O$ of origin in $\Omega$ and an arbitrary oriented line which we refer to as the $x$-axis.
Each line intersecting $\Omega$ can be associated to a pair $(\varphi, p)$, where $p$ measures the signed distance of the oriented line to the origin $O$, and $\varphi$ corresponds to the angle of its right normal with respect to the $x$-axis -- see Figure \ref{fig:oriented_line}. 

\begin{figure}[h!]
    \centering
    \begin{tikzpicture}[scale=1]

\draw[->] (-3,0) -- (3,0) node[right] {$x$};

\def\phi{40}   
\def\p{1.4}    

\coordinate (N) at ({\p*cos(\phi)},{\p*sin(\phi)});

\draw[thick, ->] ($(N)+({3*sin(\phi)},{-3*cos(\phi)})$) -- ($(N)+({-3*sin(\phi)},{3*cos(\phi)})$) node[above] {$\ell$};

\coordinate (O) at (0,0);
\fill (O) circle (1.2pt);

\draw[->,thick] (O) -- (N) node[midway,above left] {$p$};

\draw (0.6,0) arc (0:\phi:0.6);
\node at ({0.9*cos(\phi/2)},{0.9*sin(\phi/2)}) {$\varphi$};

\node[above left] at (0,0) {O};

\end{tikzpicture}
    \caption{An oriented line $\ell$ with coordinates $(\varphi,p)$: $\ell$ is at at distance $p$ from the origin $O$ and its right normal makes an angle $\varphi$ with the $x$-axis.}
    \label{fig:oriented_line}
\end{figure}
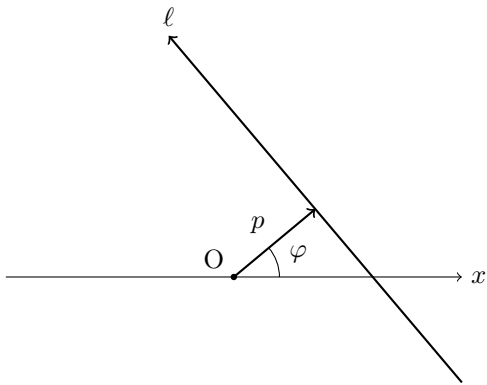

The boundary of the domain $\Omega$ is determined by its support function $h:\R/2\pi\Z\to\R$, defined so that $h(\psi)$ is the distance from $O$ to the tangent line to $\partial\Omega$ whose right normal makes an angle $\psi$ with the $x$-axis.

Following \cite[Proposition 3.1]{Bialyellipses}, the billiard map $T_{\Omega}$ in $\Omega$ is an exact symplectic twist map on the space of oriented lines intersecting $\Omega$ and its generating function is given by
\[
S(\varphi_1,\varphi_2) = -2h(\psi)\sin\delta
\]
where
\[
\psi=\frac{\varphi_1+\varphi_2}{2},
\qquad\qquad
\delta=\frac{\varphi_2-\varphi_1}{2}.
\]

\begin{definition}
Assume that the billiard map $T_{\Omega}$ in a domain $\Omega$ as an invariant curve of rotation number $\rho\in[0,1/2]$. We say that the latter is \textit{action-parametrized} by a map $\Theta\in\R/\Z\mapsto (\varphi(\Theta),p(\Theta))$ if $T_{\Omega}$ satisfies
    \[
    T_{\Omega}(\varphi(\Theta),p(\Theta)) = (\varphi(\Theta+\rho),p(\Theta+\rho)),\qquad \Theta\in\R.
    \]
Given such a parametrization, we can define
\[
\psi(\Theta)=\frac{\varphi(\Theta)+\varphi(\Theta+\rho)}{2},
\qquad\qquad
\delta(\Theta)=\frac{\varphi(\Theta+\rho)-\varphi(\Theta)}{2}.
\]
\end{definition}

\begin{remark}
    Not all invariant curves admit an action-parametrization. Although it is always the case for irrational rotation numbers and sufficiently smooth domains, this can fail for rational ones.
\end{remark}

\begin{lemma}
\label{lemma:Mather_beta}
    Assume that the billiard map in $\Omega$ as an invariant curve of rotation number $\rho\in[0,1/2]$, and action-parametrized by a map $\Theta\in\R/\Z\mapsto (\varphi(\Theta),p(\Theta))$.
    Then $\beta_{\Omega}(\rho)$ can be expressed as
    \[
    \beta_{\Omega}(\rho) = \int_0^1S(\varphi(\Theta),\varphi(\theta+\rho))d\theta = -2\int_0^1 h(\psi(\Theta))\sin\delta(\Theta)d\theta.
    \]
\end{lemma}

\begin{proof}
    Define for $q\geq 0$ and $\Theta\in\R$ the function
    \[
    \Sigma_q(\Theta) = \frac{1}{q}\sum_{k=0}^{q-1}S(\varphi(\Theta+k\rho),\varphi(\Theta+(k+1)\rho)).
    \]
    Integrating $\Sigma_q$ over $\Theta$ gives
    \[
    \int_0^1 \Sigma_q(\Theta) d\Theta = \int_0^1 S(\varphi(\Theta),\varphi(\Theta+\rho))d\Theta.
    \]
    Now since $\Sigma_q$ is uniformly bounded in $q$ and pointwise converges to $\beta_{\Omega}(\rho)$, we conclude that
    \[
    \int_0^1 \Sigma_q(\Theta) d\Theta\longrightarrow\beta_{\Omega}(\rho),\qquad q\longrightarrow+\infty
    \]
    and the result follows.
\end{proof}

In what follows, given any $r\geq 1$, we will say that a domain's boundary is $\mathscr C^r$-smooth if $h$ is $\mathscr C^r$-smooth. Given an interval $I$, a one-parameter family of domains $(\Omega_{\tau})_{\tau\in I}$ caracterized by their respective support functions $(h_{\tau})_{\tau\in I}$ is said to be $\mathscr C^r$-smooth if $(\tau,\psi)\mapsto h_{\tau}(\psi)$ is a $\mathscr C^r$-smooth map.

\begin{corollary}[See \cite{FKS_ellipse}]
\label{cor:first_var_beta}
Let $\rho\in[0,1/2)$ and $(\Omega_{\tau})_{\tau}$ be a $\mathscr C^1$-smooth family of domains, such that for any $\tau$ the domain $\Omega_{\tau}$ admits an invariant curve of rotation number $\rho$ which is action-parametrized by  $\Theta\in\R/\Z\mapsto (\varphi_{\tau}(\Theta),p_{\tau}(\Theta))$. If the map $(\tau,\Theta)\mapsto\varphi_{\tau}(\Theta)$ is $\mathscr C^1$-smooth, then $\tau\in I\mapsto\beta_{\Omega_{\tau}}(\rho)$ is $\mathscr C^1$-smooth and
\begin{equation}
\label{eq:first_var_beta}
\frac{d}{d\tau}\beta_{\Omega_{\tau}}(\rho) = \int_0^1\partial_{\tau}h_{\tau}(\psi_{\tau}(\Theta))\sin\delta_{\tau}(\Theta)d\Theta.
\end{equation}
\end{corollary}
\vspace{0.2cm}

Let $(\mathscr E_{\tau})_{\tau\in I}$ be a one-parameter of ellipses given by the implicit equations
\begin{equation}
\label{eq:ellipse}
\mathscr E_{\tau}:\quad 
\frac{x^2}{a_{\tau}^2}+\frac{y^2}{b_{\tau}^2}=1
\end{equation}
where $a_{\tau},b_{\tau}>0$ depend smoothly on $\tau\in I$. The eccentricity of $\mathscr E_{\tau}$ is the number $e_{\tau}\in[0,1]$ obtained by the formula
\[
e_{\tau} = \sqrt{1-\frac{b_{\tau}^2}{a_{\tau}^2}}.
\]
Given $\rho\in[0,1/2)$, each ellipse $\mathscr E_{\tau}$ as an invariant curve of rotation number $\rho$ which admits an action-parametrization. It corresponds to the existence of a caustic, namely a smaller ellipse given by the implicit equation
\[
\mathscr E_{\tau,\rho}:\quad 
\frac{x^2}{a_{\tau}^2-\lambda_{\tau}(\rho)}+\frac{y^2}{b_{\tau}^2-\lambda_{\tau}(\rho)}=1
\]
for a given $\lambda_{\tau}(\rho)\in [0,b_{\tau}^2)$ depending smoothly in $(\tau,\rho)$.
In what follows, we write $a_{\tau}'$ and $b_{\tau}'$ to denote the $\tau$-derivatives of $a_{\tau}$ and $b_{\tau}$.

\begin{proposition}
\label{prop:derivative_beta}
Assume that $\rho\in[0,1/2]$. There exists $C=C_{\tau,\rho}>0$ such that 
\[
\frac{d}{d\tau}\beta_{\mathscr E_{\tau}}(\rho) = 
C_{\tau,\rho} \int_0^{2\pi}\frac{a_{\tau}a_{\tau}'\cos^2\psi+b_{\tau}b_{\tau}'\sin^2\psi}{\sqrt{(1-e_{\tau}^2\sin^2\psi)(1+k_{\tau,\rho}^2\sin^2\psi)}}d\psi
\]
where $k_{\tau,\rho} = \sqrt{\frac{\lambda_{\tau}(\rho)(a_{\tau}^2-b_{\tau}^2)}{a_{\tau}^2(b_{\tau}^2-\lambda_{\tau}(\rho))}}.$
\end{proposition}

\begin{proof}
We apply Corollary \ref{cor:first_var_beta}. To simplify the notation, we remove the index $\tau$. We first consider the change of coordinates $\psi =\psi(\Theta)$ in \eqref{eq:first_var_beta},  writing $\delta(\psi)$ instead of $\delta(\Theta(\psi))$ to obtain
\[
\frac{d}{d\tau}\beta_{\mathscr E_{\tau}}(\rho) = \int_0^{2\pi}\partial_{\tau}h(\psi)\sin\delta_{\tau}(\psi)\Theta'(\psi)d\psi.
\]
Since for an ellipse $h(\psi) = \sqrt{a^2\cos^2(\psi)+b^2\sin^2(\psi)}$, we first compute
\[
\partial_{\tau} h(\psi) = \frac{a_{\tau}a_{\tau}'\cos^2\psi+b_{\tau}b_{\tau}'\sin^2\psi}{h(\psi)}.
\]
Therefore
\[
\frac{d}{d\tau}\beta_{\mathscr E_{\tau}}(\rho) = \int_0^{2\pi}\frac{a_{\tau}a_{\tau}'\cos^2\psi+b_{\tau}b_{\tau}'\sin^2\psi}{h(\psi)}\sin\delta_{\tau}(\psi)\Theta'(\psi)d\psi.
\]
If $\mu_{\rho}$ is a $T_{\mathscr E}$-invariant measure on the invariant curve of rotation number $\rho$, then the following relation holds:
\[
\Theta(\psi) = \frac{\mu_{\rho}([0,\psi])}{\mu_{\rho}([0,2\pi])}.
\]
According to \cite[Formula (3) in the proof of Theorem 4.1]{Bialyellipses}, $\mu_{\rho}$
can expressed as
\[
d\mu_{\rho} = \frac{J_{\rho}}{\sin\delta(\psi)\cos\delta(\psi)}d\psi
\]
where $J_{\rho} = \frac{\sqrt{\lambda(\rho)}}{ab}$ is the so-called Joahimsthal invariant -- see \cite[Theorem 3.3]{Bialyellipses} or \cite{fierobeCMPLX}. Hence
\[
\frac{d}{d\tau}\beta_{\mathscr E_{\tau}}(\rho) = \frac{J_{\rho}}{\mu_{\rho}([0,2\pi])}\int_0^{2\pi}\frac{a_{\tau}a_{\tau}'\cos^2\psi+b_{\tau}b_{\tau}'\sin^2\psi}{h(\psi)\cos\delta(\psi)}d\psi.
\]
Now the relation\footnote{See \cite[Section 3.2]{Bialyellipses}} $\sin\delta(\psi) = J_{\rho}h(\psi)$ implies
\begin{multline}
\cos\delta(\psi)= \sqrt{1-J_{\rho}^2h(\psi)^2} = \sqrt{1-\frac{\lambda(\rho)}{b^2}+\frac{\lambda(\rho)(a^2-b^2)}{(ab)^2}\sin^2(\psi)}\\
= \sqrt{1-\frac{\lambda(\rho)}{b^2}}\sqrt{1+k_{\tau,\rho}^2\sin^2(\psi)}.
\end{multline}
Moreover we can easily see that
\[
h(\psi) = 
a\sqrt{1-e^2\sin^2\psi}.
\]
As a consequence,
\[
\frac{d}{d\tau}\beta_{\mathscr E_{\tau}}(\rho) = 
\frac{bJ_{\rho}}{\mu_{\rho}([0,2\pi])a\sqrt{b^2-\lambda(\rho)}}
\int_0^{2\pi}
\frac{a_{\tau}a_{\tau}'\cos^2\psi+b_{\tau}b_{\tau}'\sin^2\psi}
{\sqrt{1-e^2\sin^2\psi}\sqrt{1+k_{\tau,\rho}^2\sin^2(\psi)}}d\psi
\]
and the result follows.
\end{proof}

\section{Proof of Bialy's conjecture -- Theorem \ref{thm:main_bialy_conj}}
\label{sec:proof_Bialy_conj}

In this section, two ellipses that differ only by translations or rotations are considered identical, since these transformations do not change Mather's beta function.

We begin with a proposition describing the structure of ellipses $\mathscr E$ having a prescribed value of $\beta_{\mathscr E}(\rho)$ for a given $\rho\neq 0$. We call an ellipse \textit{non-degenerate} if it is not flat, that is, if its eccentricity satisfies $e<1$.

\begin{proposition}
\label{prop:structure_ellipses_beta_fixed}
Let $\rho\in(0,1/2]$ and $c<0$. The set of non-degenerate ellipses $\mathscr E$ such that $\beta_{\mathscr E}(\rho) = c$ consists of the elements of an analytic family $(\mathscr E_e^{\rho})_{e\in[0,1)}$ where
$\mathscr E_e$ is an ellipse of eccentricity $e$.
\end{proposition}

\begin{proof}
First note that given $\rho\in(0,1/2]$, two ellipses $\mathscr E$ and $\mathscr E'$ of the same eccentricity satisfy
\[
\beta_{\mathscr E}(\rho) = \beta_{\mathscr E'}(\rho)
\]
if and only if $\mathscr E=\mathscr E'$ up to translations and rotations. Indeed, since $\mathscr E$ and $\mathscr E'$ have the same eccentricity, up to isometries, they are homothetic to eachother. Note by $\lambda>0$ the ratio to pass from $\mathscr E$ to $\mathscr E'$. It follows that 
\[
\beta_{\mathscr E}(\rho) = \beta_{\mathscr E'}(\rho) = \lambda\beta_{\mathscr E}(\rho).
\]
And since $\beta_{\mathscr E}(\rho) \neq 0$, we conclude that $\lambda=1$.

By the same property of Mather's beta function on homothetic domains, we can easily see that for any $c<0$, any $\rho\in(0,1/2]$ and any $e\in[0,1)$, there is an ellipse $\mathscr E$ of eccentricity $e$ such that $\beta_{\mathscr E}(\rho) = c$. Hence given $\rho\in(0,1/2]$, $e\in[0,1)$ and $c<0$, there is a unique ellipse $\mathscr E_e$ of eccentricity $e$ such that $\beta_{\mathscr E}(\rho) = c$.

Let us show that the family $(\mathscr E_e)_{e\in[0,1)}$ is analytic. Given $a\geq b>0$ we write 
\[
B(a,b) = \beta_{\mathscr E}(\rho)-c.
\]
given an ellipse $\mathscr E$ of semi-axes $a$ and $b$. $B$ is analytic in $(a,b)$ (see \cite{Bialyellipses}) and by Proposition \ref{prop:derivative_beta}, $DB(a,b)\neq 0$ as a $1$-form. The result follows by the implicit function theorem.
\end{proof}

The following result will prove Bialy's conjecture.

\begin{proposition}
    Let $\rho_0,\rho_1\in(0,1/2]$ be distinct, and $c<0$. Consider the analytic family of ellipses $(\mathscr E_e^{\rho_0})_{e\in[0,1)}$ defined in Proposition \ref{prop:structure_ellipses_beta_fixed} such that $\beta_{\mathscr E_e^{\rho_0}}(\rho_0)=c$ for any $e\in[0,1)$. Then the map
    \[
    e\in[0,1)\mapsto\beta_{\mathscr E_e^{\rho_0}}(\rho_1)
    \]
    is strictly monotone.
\end{proposition}

\begin{proof}
    We can assume that each $\mathscr E_e^{\rho_0}$ is described by the implicit equation
    \begin{equation}
\label{eq:ellipse_implicit_2}
\mathscr E_{e}^{\rho_0}:\quad 
\frac{x^2}{a_{e}^2}+\frac{y^2}{b_{e}^2}=1
\end{equation}
where $a_e\geq b_e>0$ depend analytically on $e$. 

Let us first assume that $\rho_0,\rho_1\neq 1/2$ so that each ellipse has invariant curves of rotation numbers $\rho_0$ and $\rho_1$. Hence the fact that $e\mapsto\beta_{\mathscr E_e^{\rho_0}}(\rho_0)$ is constant can be expressed using Proposition \ref{prop:derivative_beta} as
\[ 
\int_0^{2\pi}\frac{a_{e}a_{e}'\cos^2\psi+b_{e}b_{e}'\sin^2\psi}{\sqrt{(1-e^2\sin^2\psi)(1+k_{e,\rho_0}^2\sin^2\psi)}}d\psi=0
\]
which is equivalent to say that the non-zero vector $(a_{e}a_{e}',b_{e}b_{e}')$ is colinear to the vector $(-u,v)$ defined by
\[
u = \int_0^{2\pi}\frac{\sin^2\psi}{\sqrt{(1-e^2\sin^2\psi)(1+k_{e,\rho_0}^2\sin^2\psi)}}d\psi
\]
and
\[
v = \int_0^{2\pi}\frac{\cos^2\psi}{\sqrt{(1-e^2\sin^2\psi)(1+k_{e,\rho_0}^2\sin^2\psi)}}d\psi.
\]
Hence, again by Proposition \ref{prop:derivative_beta}, 
\begin{multline*}
\frac{d}{de}\beta_{\mathscr E_e^{\rho_0}}(\rho_1)
\propto
\int_0^{2\pi}\frac{a_{e}a_{e}'\cos^2\psi+b_{e}b_{e}'\sin^2\psi}{\sqrt{(1-e^2\sin^2\psi)(1+k_{e,\rho_1}^2\sin^2\psi)}}d\psi\\
\propto \int_0^{2\pi}\frac{v\sin^2\psi-u\cos^2\psi}{\sqrt{(1-e^2\sin^2\psi)(1+k_{e,\rho_1}^2\sin^2\psi)}}d\psi
\end{multline*}
where $\propto$ means equality up to a non-zero constant. Let us show that the latter expression does not vanish. Indeed, using the classical formula $\cos^2+\sin^2=1$ and factorizing by $u+v$ one gets
\begin{equation}
\label{eq:factorizing_dbeta}
\frac{d}{de}\beta_{\mathscr E_e^{\rho_0}}(\rho_1) 
\propto (v+u)\int_0^{2\pi}\frac{\sin^2\psi-\frac{u}{v+u}}{\sqrt{(1-e^2\sin^2\psi)(1+k_{e,\rho_1}^2\sin^2\psi)}}d\psi.
\end{equation}
Let us interpret this quantity using the functions $f_0$ and $f_1$
 defined by
 \[
 f_j(\psi) = \frac{1}{\sqrt{(1-e^2\sin^2\psi)(1+k_{e,\rho_j}^2\sin^2\psi)}},\qquad \psi\in\R,\,\, j=0,1
 \]
 associated to the measures $\mu_0$ and $\mu_1$ defined by $d\mu_j=f_j(\psi)d\psi$. Given a measurable function $g:\R\to\R$, we write, 
 \[
 \mu_j(g) = \int_0^{\frac{\pi}{2}}g(\psi)d\mu_j(\psi) = \int_0^{\frac{\pi}{2}}g(\psi)f_j\psi)d\psi.
 \]
 Using the fact that $\sin^2$ is even and $\pi$-periodic, Equation \eqref{eq:factorizing_dbeta} can be rewritten as  
 \[
\frac{d}{de}\beta_{\mathscr E_e^{\rho_0}}(\rho_1) 
\propto \mu_0(1)\mu_1(1)\left(
\frac{\mu_1(\sin^2)}{\mu_1(1)}-\frac{\mu_0(\sin^2)}{\mu_0(1)}
\right).
\]
which can be formulated using a double integral as follows:
\[
\frac{d}{de}\beta_{\mathscr E_e^{\rho_0}}(\rho_1) 
\propto
\int_0^{\frac{\pi}{2}}\int_0^{\frac{\pi}{2}}(\sin^2(y)-\sin^2(x))\left(\frac{f_1(y)}{f_0(y)}-\frac{f_1(x)}{f_0(x)}\right)f_0(x)f_0(y)dxdy.
\]
Now given two distinct $k_0, k_1> 0$, the map 
\[
\psi\in[0,\pi/2]\mapsto \sqrt{\frac{1+k_0^2\sin^2\psi}{1+k_1^2\sin^2\psi}}
\]
is strictly monotone as it is obtained by composing the square root with the strictly monotone map $\psi\mapsto \frac{1+k_0^2\sin^2\psi}{1+k_1^2\sin^2\psi} =\frac{k_0^2}{k_1^2}+ \frac{1-k_0^2k_1^{-2}}{1+k_1^2\sin^2\psi}$. As a consequence, the quotient $f_1/f_0$ is strictly monotone on $[0,\pi/2]$. Hence the quantity
\[
(\sin^2(y)-\sin^2(x))\left(\frac{f_1(y)}{f_0(y)}-\frac{f_1(x)}{f_0(x)}\right)
\]
has a constant sign on $[0,\pi/2]^2$. It follows that $\frac{d}{de}\beta_{\mathscr E_e^{\rho_0}}(\rho_1) =0$ if and only if 
\[
(\sin^2(y)-\sin^2(x))\left(\frac{f_1(y)}{f_0(y)}-\frac{f_1(x)}{f_0(x)}\right)f_0(x)f_0(y) = 0,
\]
identically, which is possible only if $f_1/f_0$ is constant, that is only when $k_{e,\rho_0}=k_{e,\rho_1}$. By the expression of $k_{e,\rho}$ given in Proposition \ref{prop:derivative_beta}, this can happen only when $\lambda(\rho_0)=\lambda(\rho_1)$, that is when $\rho_0=\rho_1$. But this is not the case, and this concludes that the map 
    \[
    e\in[0,1)\mapsto\beta_{\mathscr E_e^{\rho_0}}(\rho_1)
    \]
is strictly monotone. 
\vspace{0.2cm}

In the case when\footnote{Bialy \cite{Bialyellipses} proved the result in the case when $\rho_0=1/2$ and $\rho_1$ is rational. The following proof works also when $\rho_1$ is irrational.} one of the rotation numbers is $1/2$, say $\rho_0=1/2$, we can use \cite[Example 2 following Corollary 2.2]{Bialyellipses} which gives the formula
\[
\beta_{\mathscr E_{\mathscr E}^{\rho_0}}(1/2) = -2a_e,
\qquad e\in[0,1).
\]
Hence by assumption, $e\mapsto a_e$ is constant, and therefore $b_e=a_e\sqrt{1-e^2}$ is strictly decreasing, with $b_e'<0$. Since $\rho_1\neq 1/2$, we can apply Proposition \ref{prop:derivative_beta} to get
\[
\frac{d}{de}\beta_{\mathscr E_e^{\rho_0}}(\rho_1)
\propto
b_{e}b_{e}'\int_0^{2\pi}\frac{\sin^2\psi}{\sqrt{(1-e^2\sin^2\psi)(1+k_{e,\rho_1}^2\sin^2\psi)}}d\psi\neq 0
\]
and the result follows.
 \end{proof}

\section{Proof of Theorem \ref{thm:main_cst_perimeter}}

The proof relies on the same ideas as in the proof of Bialy's conjecture in Section \ref{sec:proof_Bialy_conj}, and we use here the same notations. Given $p>0$, we consider the analytic family of ellipses $(\mathscr E_e)_{e\in[0,1)}$ such that $\mathscr E_e$ is, up to isometries, the only ellipse of eccentricity $e$ and perimeter $p$.

Their common perimeter is given by the formula
\[
p = \int_0^{2\pi} h_e(\psi)d\psi.
\]
Differentating the latter in $e$ we obtain again that the non-zero vector $(a_{e}a_{e}',b_{e}b_{e}')$ is colinear to the vector $(-u,v)$ defined by
\[
u = \int_0^{2\pi}\frac{\sin^2\psi}{\sqrt{1-e^2\sin^2\psi}}d\psi
\]
and
\[
v = \int_0^{2\pi}\frac{\cos^2\psi}{\sqrt{1-e^2\sin^2\psi}}d\psi.
\]

Let us now fix an $\rho\in(0,1/2)$. Following the computations in Section \ref{sec:proof_Bialy_conj}, we can use Proposition \ref{prop:derivative_beta} to compute the derivative of $\beta_{\mathscr E_e}(\rho)$ and show that
\begin{multline}
\frac{d}{de}\beta_{\mathscr E_e}(\rho)
\propto
\int_0^{2\pi}\frac{a_{e}a_{e}'\cos^2\psi+b_{e}b_{e}'\sin^2\psi}{\sqrt{(1-e^2\sin^2\psi)(1+k_{e,\rho}^2\sin^2\psi)}}d\psi\\
\propto \int_0^{2\pi}\frac{v\sin^2\psi-u\cos^2\psi}{\sqrt{(1-e^2\sin^2\psi)(1+k_{e,\rho}^2\sin^2\psi)}}d\psi
\end{multline}
where $\propto$ means equality up to a non-zero constant. Now exactly in the same way as for the proof of Bialy's conjecture, one obtain
\[
\frac{d}{de}\beta_{\mathscr E_e}(\rho) 
\propto
\int_0^{\frac{\pi}{2}}\int_0^{\frac{\pi}{2}}(\sin^2(y)-\sin^2(x))\left(\frac{f_1(y)}{f_0(y)}-\frac{f_1(x)}{f_0(x)}\right)f_0(x)f_0(y)dxdy
\]
where the functions $f_0$ and $f_1$
are defined for $\psi\in\R$ by
 \[
 f_0(\psi) = \frac{1}{\sqrt{1-e^2\sin^2\psi}},\qquad
 f_1(\psi) = \frac{1}{\sqrt{(1-e^2\sin^2\psi)(1+k_{e,\rho}^2\sin^2\psi)}}.
 \]
 Note that this definition of $f_0$ corresponds to the definition of $f_0$ given in Section \ref{sec:proof_Bialy_conj} with $\rho_0=0$, since $k_{e,0}=0$. The same conclusion holds as $k_{e,\rho}\neq 0$: the derivative $\frac{d}{de}\beta_{\mathscr E_e}(\rho)$ cannot vanish as the function $f_1/f_0$ is not constant.

Hence the derivative of the analytic map $e\mapsto \beta_{\mathscr E_e}(\rho)$ never vanishes, hence is strictly monotone. Moreover by Theorem \ref{thm:bbs}, it reaches its maximum at $e=0$. hence the map is strictly decreasing, which proves the result when $\rho\neq 1/2$.
\vspace{0.2cm}

The case when $\rho=1/2$ is straightforward: if two ellipses $\mathscr E$ and $\mathscr E'$ of respective semi-axes $(a,b)$ and $(a',b')$ with $a\geq b$ and $a'\geq b'$ satisfy
\[
\beta_{\mathscr E}(1/2)=\beta_{\mathscr E'}(1/2),
\]
then $a=a'$ again by \cite[Example 2 following Corollary 2.2]{Bialyellipses}. Hence if they have the same perimeter, this forces $b'=b$ and the ellipses are identical.

\section{Local maximizers of $\beta$ - proof of Theorem \ref{thm:main_local_max}}

Fix $(\nu,\sigma)\in(0,1)\times(\frac{5}{2},+\infty)$. We define the set of \textit{$(\nu,\sigma)$-Diophantine numbers} by
\[
\mathcal D(\nu,\sigma) := \{ \rho\in (0,1/2) \,|\, \forall (m,n)\in\Z\times\Z_{>0}
\quad |n\rho-m|\geq \nu |m| n^{-\sigma}\}.
\]
Lazutkin \cite[Theorem 2]{Lazutkin_KAM} proved that there is $r>0$ such that given $C>0$ and any domain $\Omega$ of radius of curvature $\varrho$ satisfying\footnote{See \cite[Formula (1.7)]{Lazutkin_KAM}}
\[
\|\varrho\|_{\mathscr C^r}\leq C
\]
have the following property: 
there exists $b=b(C,\nu,\sigma)>0$ such that for any $\rho\in\mathcal{D}(\nu,\sigma)\cap(0,b)$, the billiard map in $\Omega$ admits an invariant curves of rotation number $\rho$. This curve admits a $\mathscr C^s$-smooth action-parametrization denoted by $\gamma_{\Omega,\rho}$, where $s=s(r,\nu\sigma)>0$ can be taken arbitrary large by increasing $r$ -- see \cite[Formula (1.6)]{Lazutkin_KAM}. Following the works of Poeschel \cite{Poeschel1, Poeschel2}, we can moreover assume that 
\[
\Omega\mapsto\gamma_{\Omega,\rho}
\]
is $\mathscr C^s$-smooth, where the topology on domains is induced by the Whitney topology on the set of radii of curvatures.
\vspace{0.2cm}

Define $r$ as given by Lazutkin theorem, corresponding to $s=1$, and we can assume that $r\geq 2$. Consider a bounded set $\mathscr S_r$ of domains with $\mathscr C^r$-smooth boundary. Let $C>0$ be such that $\|\varrho\|_{\mathscr C^r}\leq C$ for all radius of curvatures of domains in $\mathscr S_r$, and $b=b(C,\sigma,\nu)>0$ as given by Lazutkin's theorem.

Introduce $\mathscr R'$ as the set of $\rho$ which are in the set $\mathcal{D}(\nu,\sigma)\cap(0,b)$ and also satisfy
\begin{equation}
\label{eq:gutkin_nb}
\tan(n\rho) \neq n\tan(\rho),\qquad\forall n\in\Z\setminus\{0,\pm 1\}.
\end{equation}
This set has positive Lebesgue measure, as $\mathcal{D}(\nu,\sigma)\cap(0,b)$ has positive Lebesgue measure\footnote{See \cite[Formula (0.3)]{Lazutkin_KAM}}  and the set of $\rho$ satisfying \eqref{eq:gutkin_nb} is countable \footnote{See \cite[Corollary 2]{Gutkin}}. Moreover, if $\rho\in\mathscr R'$, no domains except disks can have an invariant curve of constant angle.
\vspace{0.2cm}

Let $\rho\in\mathscr R'$. Consider a strongly convex domain $\Omega$ in $\mathscr S_r$, fix an origin inside of $\Omega$ and a direction, and denote by $h$ the support function associated to $\Omega$ (see \cite{Bialyellipses}).

Consider any $\mathscr C^r$-smooth $1$-parameter family of maps $h_{\tau}:\R/2\pi\Z\to\R$ such that $h_0=h$ and 
\begin{equation}
\label{eq:def_same_perimeter}
\int_0^{2\pi}h_{\tau}(\psi)d\psi = \int_0^{2\pi}h(\psi)d\psi,
\qquad \forall\tau.
\end{equation}
For sufficiently small $\tau$, $h_{\tau}$ is the support function of a strongly convex domain $\Omega_{\tau}$ with $\mathscr C^r$-smooth boundary. By \eqref{eq:def_same_perimeter}, $\Omega_{\tau}$ has the same perimeter as $\Omega$. 
Moreover, by the aforementioned results of KAM type, it has an invariant curve of rotation number $\rho$ action-parametrized by $\gamma_{\tau}:\Theta\mapsto(\varphi_{\tau}(\Theta),p_{\tau}(\Theta))$, which is $\mathscr C^1$-smooth in $(\tau,\Theta)$.

If $\Omega$ is a local $\mathscr C^r$-maximizer or a local $\mathscr C^r$-minimizer of $\beta(\rho)$, then the map 
\[
\tau\mapsto\beta_{\Omega_{\tau}}(\rho)
\]
admits a local maximum or local minimum at $\tau=0$. By Proposition \ref{prop:derivative_beta}, it can be differentiated and we obtain
\[
\left.\frac{d}{d\tau}\right|_{\tau=0}\beta_{\Omega_{\tau}}(\rho) = \int_0^1\partial_{\tau}h_{\tau}(\psi_{\tau}(\Theta))\sin\delta_{\tau}(\Theta)d\Theta = 0
\]
where $\psi_{\tau}$ and $\delta_{\tau}$ are defined by .
Doing a change of coordinates $\psi=\psi_{\tau}(\Theta)$, we obtain the equation 
\begin{equation}
    \label{eq:first_deriv_vanish}
    \int_0^{2\pi}\left.\partial_{\tau}\right|_{\tau=0}h_{\tau}(\psi)\sin\delta_{\tau}(\psi)\Theta_{\tau}'(\psi)d\psi = 0
\end{equation}
where $\Theta_{\tau}$ denotes the inverse of $\psi_{\tau}$.
Since $(h_{\tau})_{\tau}$ can be any arbitrary $\mathscr C^r$-smooth one-parameter family of maps with a fixed average, Equation \eqref{eq:first_deriv_vanish} is equivalent to 
\[
    \int_0^{2\pi}n(\psi)\sin\delta_{0}(\psi)\Theta_{0}'(\psi)d\psi = 0
\]
for any function $n\in\mathscr C^r(\R/2\pi\Z)$ with zero average. As a consequence, one can find a constant $M\in\R$ such that
\begin{equation}
    \label{eq:first_trivialization}
    \sin\delta_{0}(\psi)\Theta_{0}'(\psi) = M,
    \qquad\forall\psi\in\R.
\end{equation}
Changing $\psi$ into $\psi_0(\Theta)$ in Equation \eqref{eq:first_trivialization}, we obtain
\begin{equation}
    \label{eq:second_trivialization}
    \sin\delta_{0}(\Theta)=M\psi_{0}'(\Theta),
    \qquad\forall\Theta\in\R.
\end{equation}
We can compute $M$ explicitely by integrating Equation \eqref{eq:second_trivialization}:
\[
M = \frac{\int_0^1\sin\delta_0(\Theta)d\Theta}{\int_0^1\psi_0'(\Theta)d\Theta} = \frac{\int_0^1\sin\delta_0(\Theta)d\Theta}{2\pi}.
\]
Therefore by Lemma \ref{lemma:Mather_beta}, 
\[
\beta_{\Omega}(\rho) = -2\int_0^{2\pi} h(\psi)\sin\delta_0(\psi)\Theta_0'(\psi)d\psi = -2M\int_0^{2\pi}h(\psi)d\psi.
\]
But since $\int_0^{2\pi}h(\psi)d\psi=|\partial\Omega|$ and using the expression of $M$ we obtain
\begin{equation}
\label{eq:nice_beta}
\beta_{\Omega}(\rho) = -2\frac{|\partial\Omega|}{2\pi}\int_0^1\sin\delta_0(\Theta)d\Theta.
\end{equation}
By Theorem \ref{thm:bbs}, 
\[
\beta_{\Omega}(\rho) \leq \frac{|\partial\Omega|}{2\pi}\beta_{\D}(\rho).
\]
Hence combining Equation \eqref{eq:nice_beta} with the expression $\beta_{\D}(\rho)=-2\sin(\pi\rho)$, we finally obtain
\[
    \sin(\pi\rho)\leq \int_0^1\sin\delta_0(\Theta)d\Theta.
\]
Using Jensen's inequality, 
\begin{equation}
\label{eq:jensen}
    \sin(\pi\rho)\leq \int_0^1\sin\delta_0(\Theta)d\Theta \leq \sin\left(\int_0^1\delta_0(\Theta)d\Theta\right).
\end{equation}
We can now compute $\int_0^1\delta_0(\Theta)d\Theta$ using the relation
\[
\psi_0(\Theta+\rho)-\psi_0(\Theta) = \delta_0(\Theta)+\delta_0(\Theta+\rho),
\qquad \Theta\in\R.
\]
Integrating on both sides we obtain
\[
\int_0^1\psi_0(\Theta+\rho)d\Theta-\int_0^1\psi_0(\Theta)d\Theta = 2\int_0^1\delta_0(\Theta)d\Theta 
\]
where the right handside is obtained doing a change of variable and using the $1$-periodicity of $\delta_0$. Moreover the same change of variable together with the fact that $\psi_0$ satisfies
\[
\psi_0(\Theta+1) = \psi_0(\Theta)+2\pi
\]
gives
\[
2\pi\rho = 2\int_0^1\delta_0(\Theta)d\Theta
\]
and hence
\[
\int_0^1\delta_0(\Theta)d\Theta = \pi\rho.
\]
Using the result in \eqref{eq:jensen}, we obtain the inequalities
\[
    \sin(\pi\rho)\leq \int_0^1\sin\delta_0(\Theta)d\Theta \leq \sin\left(\int_0^1\delta_0(\Theta)d\Theta\right) = \sin(\pi\rho).
\]
Hence equality is achieved in this application of Jensen's inequality, and therefore $\delta_0$ is constant. As a consequence $\Omega$ has a curve of constant angle associated to the rotation number $\rho$.
By construction of $\mathscr R'$ and since $\rho\in\mathscr R'$, disks are the only such billiards with this property. Hence $\Omega$ can only be a disk, that is a local maximizer as a consequence of Theorem \ref{thm:bbs}, which completes the proof.

\end{document}